\documentclass[12pt]{iopart}

\usepackage{amssymb,amscd, amsthm,youngtab}
\usepackage[dvipdfm]{graphicx}
\usepackage{pstricks,pst-node,pst-plot}

\makeatletter
\@addtoreset{equation}{section}
\makeatother
\usepackage{enumerate}
\usepackage{color,ulem}

\def\ii{{\sqrt{-1}}}

\def\tX{{\tilde{X}}}
\def\cE{{\mathcal{E}}}

\def\hell{{\hat{\ell}}}

\def\ee{\mathrm e}
\def\wg{{\mathrm w}_g}

\def\cI{{\mathcal{I}}}
\def\wdeg{{\mathrm{deg}_\mathrm{w}}}
\def\mwdeg{{\mathrm{deg}_{\mathrm{w}^{-1}}}}
\def\ldeg{{\mathrm{deg}_{\lambda}}}

\def\uab{ {w}}
\def\bs{{\bold{s}}}

\def\CC{{\mathbb C}}
\def\JJ{{\mathcal J}}
\def\hJJ{{\hat{\mathcal J}}}
\def\hkappa{{\hat{\kappa}}}
\def\TT{{\mathbb T}}
\def\hPi{{\hat \Pi}}
\def\hw{{\hat w}}
\def\ZZ{{\mathbb Z}}

\def\QQ{{\mathbb Q}}
\def\RR{{\mathbb R}}

\def\cO{{\mathcal{O}}}

\def\WW{{\mathcal{W}}}
\def\SSS{{\mathcal{S}}}

\def\nuI{{{\nu^{I}}}}
\def\hnuI{{\hat{\nu^{I}}}}
\def\nuII{{{\nu^{II}}}}
\def\Edelta{{{\alpha}}}
\def\Rdelta{{{\delta_\mathrm{R}}}}

\def\LA{\langle}
\def\RA{\rangle}
\def\hzeta{\hat{\zeta}}

\newtheorem{definition}{Definition}[section]

\newtheorem{theorem}[definition]{Theorem}
\newtheorem{proposition}[definition]{Proposition}
\newtheorem{corollary}[definition]{Corollary}
\newtheorem{remark}[definition]{Remark}
\newtheorem{lemma}[definition]{Lemma}

\def\book#1{\rm{#1}, }
\def\paper#1{\textit{#1}, }
\def\jour#1{\rm{#1}, }
\def\yr#1{({\rm{#1}) }}
\def\vol#1{\textbf{#1}}
\def\pages#1{\rm{#1}}
\def\page#1{\rm{#1}}

\def\publ#1{\rm{#1}, }
\def\by#1{{\rm{#1}, }}

\begin{document}

\title[Prime form and sigma function]{Relationship between the Prime form
and the sigma function for some cyclic (r,s) curves.}

\author{John \textsc{Gibbons},
Shigeki \textsc{Matsutani} and Yoshihiro \textsc{\^Onishi}}

\begin{abstract}
In this article, we study some cyclic $(r,s)$ curves
{
$X$ given by
}
$$y^r =x^s + \lambda_{1} x^{s-1} + \cdots +
\lambda_{s-1} x + \lambda_s.$$ We give an expression for the prime form
{
$\cE(P,Q)$, where $(P, Q \in X),$
}
in terms of the sigma function for some such curves, specifically
any hyperelliptic curve $(r,s) = (2, 2g+1)$
as well as the cyclic trigonal curve $(r,s) = (3,4)$,
$$
\cE(P,Q) =\frac{ \sigma_{\natural_{r}}(u -  v)}{\sqrt{du_1}\sqrt{d v_1}},
$$
where $\natural_r$ is a certain multi-index of differentials.
{
Here $u_1$ and $v_1$ are respectively the first components of
$u = w(P)$ and $v = w(Q)$ which are
given by the Abel map $w: X \to \CC^g$, where $g$ is the genus of $X$.
These explicit formulae are useful in applications, for instance to the 
problem of constructing classes of Schwarz-Christoffel maps to slit domains.
}
\end{abstract}

{\bf{Primary 14K25; Secondary 14H55, 20C30, 37K15.}}

{\bf{$\sigma$-function, prime form}}

\maketitle

\section{Introduction}
The Prime form is the fundamental $(-1/2, -1/2)$-form characterising a 
compact Riemann surface, which governs its geometrical properties;
let us consider a compact Riemann surface $X$ and its
universal covering $\tX$.
The prime form $E(P, Q)$ for a pair of points $P, Q \in \tX$ is defined by
$$
E(P, Q) := \frac{\theta
     \left[\begin{array}{c} \Edelta'\\ \Edelta''\end{array}\right]
               (\int^P_Q \hat \nuI)}{
               \sqrt{\zeta(P)} \sqrt{\zeta(Q)}},
$$
where $[\Edelta' \Edelta'']$ are a fixed nonsingular odd theta
characteristic, $\theta[\delta](z)$ is a Riemann theta function
and $\zeta(P)$ is
$$
\zeta(P) := \sum_{i=1}^g
\left(\partial_i\theta\left[\begin{array}{c} \Edelta'\\ \Edelta''\end{array}
          \right] (0)\right) \hat \nuI_i(P),
$$
where $\hat \nuI_i(P)$ is the vector of normalized holomorphic one forms
over $X$.
Nakayashiki gives an alternative normalisation of the prime form
such that
$$
\cE(P,Q):= \ee^{^t(w(P)-w(Q))\gamma (w(P)-w(Q)))} E(P,Q) 
$$
where {$w$ is the Abel map,} 
and $\gamma$ is a certain constant matrix for the curve,
which will be specified below.

Fay investigated the relations between the Riemann theta function and
the prime form, and found the so-called Fay trisecant formula \cite{F}.
Nakayashiki has also showed that the Kleinian sigma function for
an $(r,s)$ curve is also expressed explicitly 
in terms of the prime form \cite{N2}.

In this article, we express the prime form, in Nakayashiki's normalisation,
in terms of the sigma function, for every hyperelliptic 
$(2, 2g+1)$
curve
as well as for  a cyclic $(3,4)$ trigonal curve;
this was alluded to in the Concluding Remark in \cite{N1}.

Recently Nakayashiki \cite{N2}, as well as
Enolskii, Eilbeck and Gibbons \cite{EEG10}
have given the characterization of the sigma function in terms of the 
tau function for an algebraic curve.
Since, due to the definition of sigma function in terms of the theta function,
we have transformation formulae among theta functions, tau functions and
sigma functions, they are basically the same functions.
However the sigma and the tau functions can be written in terms of 
Schur functions, 
and both are connected to the Sato theory \cite{Mat,N2,EEG10}.
Further, unlike the theta functions, the sigma functions, 
as well as Nakayashiki's definition of the Prime
form, are modular invariant for action of Sp($g, \ZZ$).

Moreover, we will note that as a characterization of sigma, the sigma 
function could be said to be {\lq\lq}algebraic{\rq\rq} due to the 
studies \cite{B1,B2,B3,BEL,BLE,EEG10,EEL,EEMOP,KMP,Mat,MP1,MP2,O1,O2}.
In particular, in contrast to the theta functions, 
the Taylor series of the sigma function about the origin has 
coefficients which are {\it polynomials} in the parameters of the curve.

Historically, the original version of the sigma function was 
introduced by Weierstrass in order to express symmetric functions
on an affine curve associated with a Riemann surface.
Klein developed this idea further, to provide the data of the affine 
coordinate via the Jacobi inversion formulae, generalising the 
following property of the Weierstrass sigma function; the genus one 
$\sigma(u)$ function gives 
$\wp(u) = - d^2 \log \sigma(u)/ d u^2$ and $\wp(u)' = d \wp(u)/d u$
which recover the differential equation $(\wp(u)')^2 = 4 \wp(u)^3 + g_2 \wp(u) + g_3$.
In other words, by letting $y=2\wp(u)'$ and $x=\wp(u) - c$ with a certain
constant $c$, this differential equation is identified with the equation
of the affine curve:
$$y^2 = x^3 + \lambda_2 x^2 + \lambda_1 x +\lambda_0.$$

As in \cite{EEMOP}, the addition formulae for the hyperelliptic curves 
in terms of sigma functions are given by the meromorphic functions on 
the affine hyperelliptic curve, and they generalise the genus one 
addition formula:
\begin{equation}
    \frac{\sigma(u+v)\sigma(u -v)}{\sigma(u)^2\sigma(v)^2}
     = \wp(u) - \wp(v).
\label{eq:add_g1}
\end{equation}
Since for genus $g > 1$ case, the dimension of the Jacobi variety
differs from the dimension of the curve itself, some properties of the
sigma function of genus one cannot survive unchanged in the case of 
higher genus curves.
 
However recently we have discovered that for some cyclic $(r,s)$ curves,
$X$, given by equations of the form:
$$y^r =x^s + \lambda_{1} x^{s-1} + \cdots +
\lambda_{s-1} x + \lambda_s,$$
certain higher derivatives
of the sigma function provide natural generalisations
of important properties of the sigma function of genus one.
\begin{enumerate}
\item The generalized Frobenius-Stickelberger
\footnote{The formula which we call the Frobenius-Stickelberger 
relation 
seems to have been essentially discovered by Hermite (See \cite[p.458]{WW}).
However, the paper \cite{FS} is the first which is explicitly written 
in terms of the \(\sigma\) function, and we use its authors' names. 
} 
relation in Propositions
 \ref{prop:FSH} and \ref{prop:FSRTC} 
shows a simple connection between the affine coordinate ring $R_g$
and the coordinate of the Jacobi variety $\JJ_g$;
this is a natural generalization of a relation found in the elliptic case
by Frobenius and Stickelberger \cite{FS}.
\item As the Jacobi inversion formula, 
\cite{MP2} shows, 
$x = -\sigma_{\natural_1^{(1)}}(u)/\sigma_{\natural_1}(u)$ for any $(r,s)$
curve, while further, we have 
$y = (d x / d u_1)/2$ for a hyperelliptic 
$(2, 2g+1)$ curve, 
or  $y = (d x / d u_3)/3$ for a (3,4)-curve;
these can be understood as
direct generalizations of the above genus one relation.
\item Further, as in Propositions \ref{lm:add2H} and \ref{prop:add2T}, 
we have a more direct generalization of (\ref{eq:add_g1}).
Hence we believe that it is very important to rewrite the 
formulae of a Riemann surface in terms of sigma functions because
they are more directly connected with
algebraic properties of the affine ring $R_g$.
\end{enumerate}

The paper is organised as follows. In the next section, we will review the
theory of cyclic $(r,s)$ curves, introducing the key ideas and fixing our
notation. Section 3 generalises Mumford's construction of meromorphic functions
on hyperelliptic curves to general cyclic $(r,s)$ curves; here we introduce
the Frobenius-Stickelberger matrices and their determinants. Section 4 reviews
the $\sigma$-function and its properties. The following section is concerned
in particular with the vanishing properties of $\sigma$ and certain of its
derivatives. From these results, it is easy to restate the generalisation
of the Frobenius-Stickelberger relation to general hyperelliptic curves,
as well as the simplest cyclic trigonal curves, 
which is done in section 6. 
In section 7, the Prime form is introduced; 
the Frobenius-Stickelberger results are
then used to give a simple expression for this in terms of derivatives of
$\sigma$. One of the motivations of this study was to permit the 
explicit construction of Schwartz-Christoffel maps giving formulae for 
certain reductions of the Benney hierarchy \cite{BG1, BG2, G}; 
we will discuss this briefly in the last section. 

Recently Nakayashiki \cite{N3,NY} has also independently generalized 
the expression of the prime form in terms of the sigma function to 
every $(r,s)$ curve, though his proof, and the form of his results, 
differ in detail from those given here.  
As the general $(r,s)$ curve lacks the cyclic symmetry 
of the curves studied here, the results in the cyclic 
case can be expressed in a different, and in some ways simpler, form; 
our results explicitly use this
symmetry. As the results for cyclic curves have applications to problems
such as conformal mappings and the reductions of Benney's equations, 
it is worth stating these separately.

\bigskip

This work was started at an international conference 
in Hanse Wissenschaftskolleg, Delmenhorst,  
organised in 2011 by Claus Laemmerzahl, 
Jutta Kunz, Emma Previato and Victor Enolskii. 
Two of the authors (J.G and S.M.) wish to thank 
HWK and the organisers for their hospitality.
After finishing this work, we heard of Nakayashiki's work on a related problem
in \cite{N3,NY}. 
We would like to thank him for discussing his work on this problem with us.

\section{Cyclic $(r,s)$ curves}\label{curves}

We briefly summarise the  conventions and  notations used in this article.
Defining a Riemann surface as in \cite[p.31]{ACGH},
we consider a cyclic $(r,s)$ Riemann surface
$$
        X:=\{(x,y) \ | \ y^r = f(x)\}\cup \infty
$$ 
whose finite part is given by an affine equation
\begin{equation}
y^r = f(x), \quad f(x) := x^s + \lambda_{1} x^{s-1} + \cdots +
\lambda_{s-1} x + \lambda_s.
\end{equation}
The positive integers $r$ and $s$ are coprime, with $r<s$; 
the complex numbers 
$\lambda_1,\ldots ,
\lambda_{s}$ are such that the finite part of $X$ is smooth and
has geometric genus $g=\frac{(r-1)(s-1)}{2}.$
In this article, we consider all hyperelliptic curves, $(r,s) = (2, 2g+1)$
as well as the cyclic trigonal curve $(r,s) = (3, 4)$.

Let $R_g:=\CC[x,y]/(y^r - f(x))$,
$\cO_X$ be the sheaf of  holomorphic functions over $X$
 and $\JJ$ the Jacobian of $X$. \label{pg:Jacobian}
$R_g=\cO_X(*\infty) $ is the ring of meromorphic functions  
on $X$ regular outside $\infty$.

We have the (monic) monomial
 $\phi_n \in R_g$ \label{pg:phi_n} for a non-negative integer $n$ so 
  that it has order of pole  $N(n)$ \label{pg:N(n)}
at $\infty$, the $n$-th integer
 in  the (increasing)  sequence complementary to the 
Weierstrass gaps:
  $\phi_0 = 1$, $\phi_1 = x$, etc.;
by letting  $t_\infty$ be a local parameter at $\infty$,
the leading term of $\phi_n$ is proportional to $t_\infty^{-N(n)}$.
By elementary computations, we see that
$N(0) = 0$,
$N(g-1) = 2g-2$,
$N(g) = 2g$ for $(r,s)$ curve.
Let $s_n$ and $r_n$ be such that $\phi_n = x^{s_n} y^{r_n}$
and then $N(n) = s_n r + s r_n$.

We introduce the w-degree, $\wdeg : R_g \to \ZZ$, \label{pg:wdeg}
which assigns to any element of $R_g$  
the order of its pole at $\infty$, 
$\wdeg(x) = r$, $\wdeg(y) = s$,
$\wdeg(\phi_n(P)) = N(n)$. 
 Further we also  consider the ring
$R_{g,\lambda}:=\QQ[x, y, \lambda_0, \ldots, \lambda_{s-1}]/(y^r - f(x))$
by regarding $\lambda$'s as indeterminates, and define a
 $\lambda$-degree,
$\ldeg:R_{g,\lambda} \to \ZZ$ as an extension of the w-degree by
assigning the degree $j r$ to  each $\lambda_j$.
The defining polynomial $y^r-f(x)$ of the curve 
is homogeneous with respect to the $\lambda$-degree.

For later convenience, we sometimes denote a point
$P \in X\backslash\infty$ by its affine coordinates $(x, y)$.
For the $k$-th symmetric product of the curve $\SSS^k(X)$, 
its element is sometimes expressed by
$(P_1,\ldots ,P_k )$ and a divisor $D=\sum_{i=1}^kP_i$.

For a local parameter $t$ at  $P$ in $X$, we denote by
$d_>(t^\ell)$  the  terms of the $t$-expansion of
a function on $X$ whose orders of zero at $P$
are greater than $\ell$ \label{pg:d_>}.

\medskip

For  the canonical bundle $K_X$  
a basis $\{\nuI_1, \ldots, \nuI_g\}$ of $H^0(X, K_X)$
is given in terms of the $\phi_i$ following 
\cite[Ch. VI, \S91]{B1},
\begin{equation}
\nuI_i:=
 \frac{\phi_{i-1}(P) d x}{r y^{r-1}}, \quad(i = 1, \ldots, g).
\label{eq:1stkind}
\end{equation}
For a hyperelliptic curve, these are
\begin{equation}
\nuI_1= \frac{d x}{2 y}, \quad
\nuI_2= \frac{x d x}{2 y}, \quad \cdots, \quad
\nuI_g= \frac{x^{g-1}d x}{2 y}, \quad
\label{eq:1stkindHC}
\end{equation}
and for the $(3,4)$ curve case,
\begin{equation}
\nuI_1= \frac{d x}{3 y^2}, \quad
\nuI_2= \frac{xd x}{3 y^2}, \quad
\nuI_3= \frac{d x}{3 y}. \quad
\label{eq:1stkind34}
\end{equation}

For a one-form $\nu$ which is 
expressed by  $(t_\infty^n + d_>(t_\infty^n)) d t_\infty$ 
in terms of the local parameter $t_\infty$ at $\infty$, 
we extend the w-degree to one-forms
$\wdeg(\nu) = - n$.
Since 
\begin{equation}
 \nuI_i = t_\infty^{2g-N(i-1)-2}(1+d_{>0}(t_\infty)) d t_\infty,
\label{eq:dnu_tinf}
\end{equation}
we have
$$
 \mwdeg(\nuI_i) = 2g-N(i-1)-2,
$$
where $\mwdeg(f) = - \wdeg(f)$.

We take a homology basis 
$   \alpha_i, \beta_j$  $ (1\leqq i, j\leqq g)$
of $H_1(X,\ZZ)$ with intersection pairing
$\alpha_i\cdot\alpha_j=\beta_i\cdot\beta_j= 0$,
$\alpha_i\cdot\beta_j=\delta_{ij}$.
We have the half period matrices,
\begin{equation}
   \left[\,\omega'  \ \omega''  \right]= 
\frac{1}{2}\left[\int_{\alpha_i}\nuI_j \ \ \int_{\beta_i}\nuI_j
\right]_{i,j=1,2, \ldots, g},
\label{eq:omega}
\end{equation} 
and the period lattice $\Pi$ generated by
$\left[\,2\omega'  \ 2\omega''  \right]$.
Since the Jacobian $\JJ$ of $X$ is given by 
$\JJ= \CC^g/\Pi$, we have a natural projection $\kappa : \CC^g \to \JJ$.

For a positive integer $k$, 
the Abel map $\uab : \SSS^k(X) \to \CC^g$ is defined by
$\uab:(P_1,\ldots ,P_k)\mapsto \uab(P_1,\ldots ,P_k)
=\sum_{i=1}^k\int_{\infty}^{P_i}\nuI\in\mathbb{C}^g$
with base-point $\infty$:
\label{pg:Abelmap}
\begin{equation}
    \WW^k := \kappa\, \uab(\SSS^k(X)) \subset\JJ .
\label{eq:W_k1}
\end{equation}
We note that the Abel map $\uab$ has a 
$\Pi$-ambiguity coming from the choice of path of integration.

For later convenience, we also introduce the normalized
periodic matrix  $\left[1 \  \TT \right]$ with
$\TT:={\omega'}^{-1} \omega''$,
the normalized one form $\hnuI := {\omega'}^{-1} \nuI$, 
the normalized period lattice $\hPi$ generated by $\left[1 \  \TT \right]$,
the normalized Jacobian  $\hJJ:=\CC^g/\hPi$,
the natural projection $\hkappa : \CC^g \to \hJJ$,
and the normalized Abel map $\hw(P) := {\omega'}^{-1} w(P)$.
Since $X$ is a non-singular curve,  
${\omega'}$ is a non-singular matrix, and hence
the correspondence  between $\hJJ$ and $\JJ$ is a bijection.

We also introduce  the singular locus,
$$
\SSS^n_m(X) := \{D \in \SSS^n(X) \ | \
    \mathrm{dim} | D | \ge m\},
$$
where $|D|$ is the complete linear system
$\uab^{-1}(\uab(D))$ 
\cite[IV.1]{ACGH}. \label{pg:S^k(X)}
If $n< g$, the singular locus of $\SSS^n(X)$ 
modulo linear equivalence, or on
 projecting to the Picard group,
is $\SSS_1^n(X) $ 
 \cite[Ch. IV, Proposition 4.2, Corollary 4.5,
where our $\SSS^n(X)$ is $C^0_n$]{ACGH}. We let 
$\mathcal{W}^n_m:=\kappa w(\mathcal{S}^n_m (X))$.
\label{pg:W^n_m}

\medskip

Since the sigma function is connected with the $\tau$ function
in Sato theory due to \cite{N1,EEG10} and the $\tau$ function
is written for a Young diagram,
we construct a Young diagram (cf., e.g., \cite{Sa, BEL})
$\Lambda$ from the Weierstrass gap sequence: \label{pg:Lambda}
from the top  down, $1\le i\le g$, the rows have length:
$$
        \Lambda_i = N(g) - N(i-1)  -g + i -1=g-N(i-1)+(i-1).
$$
We define $\Lambda_j\equiv 0$ for all $j>g$.
It is known that  for any $(r,s)$ curve, we have \cite{Sa, BEL})
$$
        |\Lambda| = \sum_{i=1}^g \Lambda_i = \frac{1}{24}(r^2 - 1) (s^2 -
        1).
$$

We give two examples: 
For the case  $(r, s) = (2, 9)$ (Table 2.1), 
and the case  $(r, s) = (3, 4)$ (Table 2.2), 
we have
\begin{equation*}
{
\vbox{
        \baselineskip =10pt
        \tabskip = 1em
        \halign{&\hfil#\hfil \cr
        \multispan7 \hfil Table 2.1 \hfil \cr
        \noalign{\smallskip}
        \noalign{\hrule height0.8pt}
        \noalign{\smallskip}
$i$ & \strut\vrule& 0 &1 & 2 & 3 & 4 & 5 & 6 & 7 \cr
\noalign{\smallskip}
\noalign{\hrule height0.3pt}
\noalign{\smallskip}
$\phi(i)$ & \strut\vrule 
&1& $x$& $x^2$ & $x^3$ &$y$ & $x^4$& $xy$ & $x^5$ \cr
$N(i)$ &
 \strut\vrule & 
 0&  2 & 4 & 6 & 8 & 9 & 10 & 11 &  \cr 
$\Lambda_i$ &
 \strut\vrule & 
-  & 4 & 3 & 2 & 1 & - & - & -  \cr 
\noalign{\smallskip}
\noalign{\hrule height0.3pt}
\noalign{\smallskip}
        \noalign{\hrule height0.8pt}
}
}
}
\quad
 \yng(4,3,2,1), \quad
\end{equation*}
\begin{equation*}
{
\vbox{
        \baselineskip =10pt
        \tabskip = 1em
        \halign{&\hfil#\hfil \cr
        \multispan7 \hfil Table 2.2 \hfil \cr
        \noalign{\smallskip}
        \noalign{\hrule height0.8pt}
        \noalign{\smallskip}
$i$ & \strut\vrule& 0 &1 & 2 & 3 &  \cr
\noalign{\smallskip}
\noalign{\hrule height0.3pt}
\noalign{\smallskip}
$\phi(i)$ & \strut\vrule 
&1& $x$& $y$ & $x^2$ \cr 
$N(i)$ &
 \strut\vrule & 
 0&  3 & 4 & 6  \cr 
$\Lambda_i$ &
 \strut\vrule & 
-  & 3 & 1 & 1  \cr 
\noalign{\smallskip}
\noalign{\hrule height0.3pt}
\noalign{\smallskip}
        \noalign{\hrule height0.8pt}
}
}
}
  \quad
 \yng(3,1,1). \quad
\end{equation*}

For a given Young diagram $\Lambda=(\Lambda_1, \Lambda_2, \ldots, 
\Lambda_\ell)$, 
the length $r$ of the diagonal is
called the {\it rank} of the partition \cite[\S 4.1, p. 51]{FH}.
Let $a_i$ and $b_i$ be the number of boxes below and to the 
right of the $i$-th box of the diagonal, reading from lower right
to upper left.
Frobenius called
$(a_1, \ldots, a_r; b_1, \ldots, b_r)$
{\it{ the characteristics of
the partition}} \cite[\S 4.1 p. 51]{FH}. \label{pg:charac.part.}
Here $a_i < a_j$ and $b_i < b_j$ for $i < j$.

\begin{lemma} For $u\in w(P)$, $P\in X$,
$$
        \mwdeg(u_i) = N(g) - N(i-1) - 1  
         = 2g - N(i-1) - 1 = \Lambda_i + g - i.
$$
\end{lemma}

Here we also extend the degree to $v = w(P_1,\ldots,P_k)$
for points $P_i \in X$ ($i=1,\ldots, k)$ by
$\mwdeg(v_i) = \mwdeg(u_i)$ since by letting each $P_j$ be expressed
as a local parameter $t_{\infty,j}$ around $\infty$, we  have, formally:
\begin{equation}  
 v_i = 
\frac{1}{2g-N(i-1)-1} 
\left(t_{\infty,1}^{2g-N(i-1)-1}+
\cdots
+t_{\infty,k}^{2g-N(i-1)-1}\right)
\left(1+d_{>0}(t_{\infty})\right).
\label{eq:v_t}
\end{equation} 
Here we use the multi-index convention
for $\alpha := (\alpha_1, \ldots, \alpha_g)$,
\label{pg:multi-index}
$$
        t^\alpha := t_1^{\alpha_1} t_2^{\alpha_2}
        \cdots t_g^{\alpha_g}, \quad
        \alpha = ( \alpha_1, \alpha_2, \ldots, \alpha_g),
        \quad |\alpha | := \sum_{i=1}^g  \alpha_i,
$$
and extend the definition $d_{>}(t_i^\ell)$
to the variables, $t_1,\ldots,t_g$:
$d_{>}(t^\ell)\in \{\sum_{|\alpha|>\ell} a_\alpha t^\alpha\}$.

It should be noted that
$v_g = 
\left(t_{\infty,1}+ \cdots +t_{\infty,k}\right)
\left(1+d_{>0}(t_{\infty})\right)$ and $\mwdeg(v_i) = 1$.

By Serre duality, $\dim H^0(X, D)=\deg D-g+1-
\dim H^0(X, (2g-2)\infty -D)$, we find that the Young diagram is symmetric;
the characteristics $a_i=b_i$.
We also have the fact that
$ \mwdeg(u_i) = \Lambda_i + g - i$ is the hooklength  (cf. \cite[Ch. 3]{Sa})
of the node $(1,i)$ in the Young diagram $\Lambda$.

For later convenience, we introduce  `truncated Young diagrams' 
\begin{equation}
\Lambda^{(k)}:= (\Lambda_1, \ldots, \Lambda_{k}), \quad
\Lambda^{[k]}:= (\Lambda_{k+1}, \ldots, \Lambda_{g}).
\label{eq:Lambdak}
\end{equation}

\begin{remark}\label{rmk:sigma_zeta}{\rm{
{\bf{Galois action on $X$:}}
Since there is an action of the cyclic group $\mathbb{Z}_r$ on the
curve $X$ such that for $\hzeta_r \in \mathbb{Z}_r$,
$\hzeta_r^i: (x, y) \mapsto (x, \zeta_r^i y)$ for a 
fixed primitive \(r\)-th root 
$\zeta_r$ of 1. 
Naturally we have the induced actions;
$\hzeta_r \phi_n = \zeta_r^{r_n} \phi_n$,
{
$\hzeta_r \nuI_n = \zeta_r^{r_{n-1}-r+1} \nuI_n = \zeta_r^{r_{n-1}+1} \nuI_n$ 
and
for $u \in \WW^k$,
$\hzeta_r u_n = \zeta_r^{r_{n-1}+1} u_n$.
}

{
For example, for the case of the $(3,4)$ curve,
$(r_0, r_1, r_2) = (0, 0, 1)$ and hence
$\hzeta_3 u_1 = \zeta_3 u_1$,
$\hzeta_3 u_2 = \zeta_3 u_2$,
$\hzeta_3 u_3 = \zeta_3^2 u_3$. 
}
}}
\end{remark}

\section{Meromorphic function over $(r,s)$ curve}

In \cite{MP1},
we introduced meromorphic functions on the curve,
as a generalization of the polynomial  
$U$ in Mumford's $(U,V,W)$ parameterization of
a hyperelliptic Jacobian \cite[Ch. IIIa]{Mu}.

\medskip

We introduce the Frobenius-Stickelberger (FS) matrix
and its determinant following \cite{MP1}.
Let $n$ be a positive integer  and 
$P_1, \ldots, P_n$ be in $X\backslash\infty$.
We  define the \textit{$\ell$-reduced  
Frobenius-Stickelberger} (FS) \textit{matrix} by: \label{pg:FS-matrix}
\begin{equation*}
\Psi_{n}^{(\check\ell)}(P_1, P_2, \ldots, P_n) := 
\left(
\begin{array}{ccccccc}
1 &\phi_1(P_1) & \phi_2(P_1)  &\cdots & \check{\phi_{\ell}}(P_1) 
& \cdots & \phi_{n}(P_1) \\
1 & \phi_1(P_2) & \phi_2(P_2) &\cdots & \check{\phi_{\ell}}(P_2) 
 & \cdots & \phi_{n}(P_2) \\
\vdots & \vdots & \vdots & \ddots& \vdots & \ddots& \vdots\\
1 & \phi_1(P_{n}) & \phi_2(P_{n}) &\cdots & \check{\phi_{\ell}}(P_n) 
 & \cdots&  \phi_{n}(P_{n})
\end{array}\right),
\end{equation*}
and $\psi_{n}^{(\check\ell)}(P_1, P_2, \ldots, P_n) := 
| \Psi_{n}^{(\check\ell)}(P_1, P_2, \ldots, P_n)|$
(a check on top of a letter signifies deletion).
It is also convenient to introduce the simpler notation:
\begin{equation}
\psi_{n}(P_1, \ldots, P_n)
 := |\Psi_{n}^{(\check{n})}(P_1, \ldots, P_n)|, \quad
\Psi_{n}(P_1, \ldots, P_n):= \Psi_{n}^{(\check{n})}(P_1, \ldots, P_n),
\label{eq:psin}
\end{equation}
for the un-bordered matrix.
We call this matrix the {\it{Frobenius-Stickelberger (FS) matrix}}
and its determinant the {\it{Frobenius-Stickelberger (FS) determinant}}.
These become undefined for some tuples in $(X \backslash \infty)^n$.

Further for later convenience, we will define
$$
\Phi_{n}(P_1, P_2, \ldots, P_n) := 
\left(
\begin{array}{ccccc}
1 &x_1 &x_1^2  &\cdots  & x_1^{n-1} \\
1 &x_2 &x_2^2  &\cdots  & x_2^{n-1} \\
\vdots & \vdots & \vdots & \ddots& \vdots\\
1 &x_n &x_n^2  &\cdots  & x_n^{n-1} \\
\end{array}\right),
$$
and
$\varphi_{n}(P_1, \ldots, P_n)
 := |\Phi_{n}^{(\check{n})}(P_1, \ldots, P_n)|$.

\medskip

For $n$ points $(P_i)_{i=1, \ldots, n}$ $\in X\backslash\infty$,
we find an element $\mu_n(P) :=\mu_n(P; P_1, \ldots, P_n)
:= \sum_{i=0}^{n} a_i \phi_i(P)$
 of $R_g$ associated with 
a point $P= (x,y)$ in $(X\backslash\infty )$,
$a_i \in \CC$ and $a_n = 1$,
such that $\mu_n(P)$ has a  zero at each point $P_i$ (with multiplicity, if
the $P_i$ are repeated)
and has smallest possible order of pole at $\infty$.
We obtain $\mu_n(P)$ using the FS matrix:

\begin{proposition} \label{prop:mul}
For $P, P_1, \ldots, P_n$ $\in (X\backslash\infty) \times \SSS^n(X\backslash\infty)$,
we have $\mu_n(P)$ by
$$
\mu_n(P)\equiv 
\mu_n(P; P_1,  \ldots, P_n) = 
\lim_{P_i' \to P_i}\frac{1}{\psi_{n}(P_1' , \ldots, P_n' )}
\psi_{n+1}(P_1' , \ldots, P_n' , P),
$$
where the $P_i^\prime$ are generic and
the limit is taken (irrespective of the order) for each i.
\end{proposition}

We introduce $\mu_{n, k}(P_1, \ldots, P_n)$ by
$$
\mu_n(P)
 = \phi_n(P) + 
\sum_{k=0}^{n-1} (-1)^{n-k}\mu_{n, k}(P_1, \ldots, P_n) \phi_k(P).
$$
Proposition \ref{prop:mul} gives the following lemma:
\begin{lemma}\label{prop:2theta2}
Let $n$ be a positive integer.
For $(P_i)_{i=1,\ldots, n}\in \SSS^n(X \backslash\infty) $,
the  function  $\mu_n$
over $X$ induces the map: 
$$
\alpha_n: 
\SSS^n(X \backslash\infty)  \to \SSS^{N(n) - n}(X), 
$$
{{i.e.}},
$(P_i)_{i=1,\ldots, n}\in \SSS^n(X \backslash\infty) $ is mapped to
an element $(Q_i)_{i=1,\ldots, N(n)-n} \in \SSS^{N(n)-n}(X)$, such that
$$
\sum_{i=1}^{n} P_i - n \infty 
\sim - \sum_{i=1}^{N(n)-n} Q_i  + (N(n)-n) \infty .
$$
\end{lemma}

\bigskip
For an effective divisor 
of degree $n$, $D\in \SSS^n(X)$, 
let $D'$  be the maximal subdivisor  of $D$ which does
not contain $\infty$,
$D = D' + (n-m) \infty $
where $\deg D'=m (\le n)$ and
$D' \in \SSS^m(X\backslash\infty)$:
we extend the map  ${\alpha}_n$ to $\SSS^n(X)$ 
by defining $\overline{\alpha}_n(D)={\alpha}_m(D^\prime )+[N(n)-n-
(N(m)-m)]\infty.$

The linear equivalence of Lemma \ref{prop:2theta2} gives:
\begin{proposition} \label{prop:addition}
For a positive integer $n$, the
Abel map composed with $\alpha_n$ induces the map 
$[-1] :  \WW^n \to  \WW^{N(n) - n}$, $(u \mapsto -u) $ 
satisfying the following commutative diagram;
$$
\begin{CD}
S^n(X -\infty) @>{\alpha_n}>> S^{N(n) - n}(X) \\
  @V{\kappa \circ w}VV @V{\kappa \circ w}VV \\
\WW^n @>{[-1]}>> \WW^{N(n) - n} \\
\end{CD},
$$
\end{proposition}

Let $\mathrm{image}([-1])$ be denoted by $[-1]\WW^n$.

\bigskip
For $(2, 2g+1)$ curve $y^2 = x^{2g+1} + \lambda_1 x^{2g} + 
\cdots + \lambda_{2g} x + \lambda_{2g+1}$,
the $[-1]$ operation corresponds to  the hyperelliptic involution  
$(x, y) \to (x, -y)$. This correspondence does not hold for 
cyclic $(r,s)$ curves in general.

\bigskip
For $(3,4)$
 curve,
we consider $[-1](x_1, y_1)$.
By considering the divisor of a meromorphic function of $(x, y)$,
\begin{equation*}
\mu((x,y);(x_1,y_1))\equiv
\left|\begin{array}{cc}
1 & x_1 \\
1 & x \\
\end{array}\right| =0.
\end{equation*}
This means that 
$$ 
   (x_1, y_1) + (x_1,\zeta y_1) + (x_1,\zeta^2 y_1) -3\infty \sim 0. 
$$
In other words, we see that
\begin{equation}
[-1] (x_1, y_1) =  (x_1,\zeta y_1) + (x_1,\zeta^2 y_1).  
\label{eq:[-1]}
\end{equation}
Hence, we must have
$$
  \int_{\infty}^{(x_1, y_1)} 
   \left(\begin{array}{c}
       \nuI_1\\ \nuI_2\\ \nuI_3
     \end{array}\right)
  + \int_{\infty}^{(x_1, \zeta y_1)} 
   \left(\begin{array}{c}
\zeta^2 \nuI_1\\ \zeta^2 \nuI_2\\ \zeta \nuI_3
     \end{array}\right)
  +  \int_{\infty}^{(x_1, \zeta^2 y_1)} 
   \left(\begin{array}{c}
     \zeta \nuI_1\\ \zeta \nuI_2\\ \zeta^2 \nuI_3
     \end{array}\right)
\equiv 0, \quad \mbox{modulo } \Pi.
$$

\medskip
\section{The $\sigma$-function}\label{the sigma function}

As in (\ref{eq:omega}), we also introduce the
complete Abelian integral of the second kind,
\begin{equation}
   \left[\,\eta'  \ \eta''  \right]= 
\frac{1}{2}\left[\int_{\alpha_i}\nuII_j \ \ \int_{\beta_i}\nuII_j
\right]_{i,j=1,2, \ldots, g},
   \label{eq2.5}
  \end{equation} 
where $\nuII_j=\nuII_j(x,y)$ $(j=1, 2, \cdots, g)$ 
are  differentials of the 
second kind  \cite[Corollary 2.6]{F}, which we defined algebraically
in \cite{MP1} after \cite{EEL}.

The following Proposition gives a 
  generalized Legendre relation is given as 
the following Proposition \cite{B1, BLE, EEL}.
\begin{proposition} \label{prop:gLegendreR}
The matrix
\begin{equation}
   M := \left[\begin{array}{cc}2\omega' & 2\omega'' \\ 2\eta' & 2\eta''
     \end{array}\right],
\end{equation} 
 satisfies 
\begin{equation}
   M\left[\begin{array}{cc} & -1 \\ 1 & \end{array}\right]{}^t {M}
   =2\pi\sqrt{-1}\left[\begin{array}{cc} & -1 \\ 1 &
     \end{array}\right].
   \label{eq2.7}
\end{equation} 
\end{proposition}

By the  Riemann bilinear relations \cite{F}, it is known that
 $\mbox{Im}\,({\omega'}^{-1}\omega'') $ is positive definite.
Referring to  Theorem 1.1 in \cite{F}, let
\begin{equation}
   \Rdelta:=\left[\begin{array}{cc}\Rdelta'\ \\
       \Rdelta''\end{array}\right]
\in \left(\frac{1}{2}\ZZ\right)^{2g}
   \label{eq2.9} 
\end{equation} 
be the theta characteristic which gives the Riemann-constant vector
$\omega_\mathrm{R}=2\omega^{\prime\prime}\Rdelta^{\prime}
+2\omega^\prime\Rdelta^{\prime\prime}$
 with respect to the base point $\infty$ and the period matrix 
$[\,2\omega'\ 2\omega'']$. 

 We define the $\sigma$ function as an entire function of 
$u={}^t\negthinspace (u_1, u_2, \ldots, u_g)\in \mathbb{C}^g$,
\begin{equation}
\begin{array}{rl}
   \sigma(u)&=\sigma(u;M)=\sigma(u_1, u_2, \ldots, u_g;M) \\
   &=c\,\mathrm{exp}(-\frac{1}{2}{}\ ^t\negthinspace  u\gamma u)
   \theta\negthinspace
   \left[\Rdelta\right](\frac{1}{2}{\omega'}^{-1} u;\ 
{\omega'}^{-1}\omega''), \\
\end{array}
   \label{def_sigma}
\end{equation}
where  
$\gamma := \omega^{\prime -1} \eta'$,
$c$ is a constant that depends on the moduli of the curve
and $\theta$ is  the Riemann $\theta$ function of $\delta 
   :=\left[\begin{array}{cc}\delta'\ \\
       \delta''\end{array}\right]\in \left(\frac12\ZZ\right)^{2g}$,
\begin{equation*}
   \theta\negthinspace
   \left[\delta\right](z ; \TT)
   =
   \sum_{n \in \ZZ^g} \exp\big\{- \sqrt{-1}\pi \big[
    \ ^t\negthinspace (n+\delta')\TT(n+\delta')
   + 2\ ^t\negthinspace (n+\delta')(z+\delta'')\big]\big\}.
   \label{def_theta}
\end{equation*}

For a given $u\in\CC^g$, we  introduce
$u'$ and $u''$ in $\RR^g$ so that $u=2\omega'u'+2\omega''u''$.
Further for $u$, $v\in\CC^g$, and $\ell$
($=2\omega'\ell'+2\omega''\ell''$) $\in\Pi$, 
we define
\begin{equation*}
\begin{array}{rl}
  L(u,v)    &:=2\ {}^t{u}(\eta'v'+\eta''v''),\nonumber \\
  \chi(\ell)&:=\exp[\pi\sqrt{-1}\big(2({}^t {\ell'}\Rdelta'-{}^t
  {\ell''}\Rdelta'') +{}^t {\ell'}\ell''\big)] \ (\in \{1,\,-1\}).
\end{array}
\end{equation*}
Then we have the following quasi-periodic properties of the $\sigma$ function
 \cite[Proposition 4.3]{MP2};
\begin{equation}
\sigma(u + \ell) = \sigma(u) \exp(L(u+\frac{1}{2}\ell, \ell)) \chi(\ell).
        \label{eq:4.11}
\end{equation}

This quasi-periodic property of $\sigma$ is a straightforward consequence
of the similar property holding for the normalized theta function in Chapter VI of \cite{L},
namely, for $z \in \CC^g$ and $\ell', \ell'' \in \ZZ^g$,
$$
\theta\left[\delta\right]\left(z + \ell' + \tau \ell''\right) 
=
\ee^{-2\pi\ii( \frac{1}{2} ^t\ell^{\prime\prime}\tau \ell''
+^t\ell'' z +^t \ell'\delta'' -^t \ell''\delta')} 
\theta\left[\delta\right]\left(z\right). 
$$
However, we wish to stress that one significant difference
between $\theta$ and $\sigma$ is that unlike $\theta$, $\sigma$ is also 
{\it modular invariant}: that is, for all  $u\in{\mathbf C}^g$, 
and $\gamma\in\mathrm{Sp}(2g,{\mathbf Z})$,
we have also:
$$
\sigma(u;\gamma M)=\sigma(u;M).
$$

The vanishing locus of $\sigma$ is:
\begin{equation}
        \Theta^{g-1} =( \WW^{g-1} \cup [-1] \WW^{g-1}) =
\WW^{g-1}.
\label{eq:Theta:g-1}
\end{equation}
The last equality is due 
to our choice of base point $\infty$ such that
$(2g-2)\infty=K_X$.
Since $\WW^{k}$ is not $[-1]$-invariant in general for $k<g-1$,
we define:
\begin{equation}
        \Theta^{k} := \WW^{k} \cup [-1] \WW^{k},
\label{eq:Theta:k}
\end{equation}
and similarly,
\begin{equation}
        \Theta^{k}_1 := w(\SSS^{k}_1 (X)) \cup [-1] w(\SSS_1^{k}(X)).
\label{eq:Theta:k1}
\end{equation}
For  hyperelliptic curves $(r=2,s=2g+1)$ with a branch point at 
$\infty$,  $\Theta^k$ 
equals  $\WW^k$ for every positive integer $k$,
but for general curves it does not.

The Schur function $\bold{s}_\Lambda (t)$ is given by
\begin{equation}
        \bold{s}_{\Lambda}(t) :=\frac{
          |t_j^{\Lambda_i + g - i}|_{1\le i, j \le g}}
          {|t_j^{i-1}|_{1\le i, j \le g}},
\label{eq:def_Schur}
\end{equation}
and the complete homogeneous symmetric function 
$h_n^{\LA\ell_1,\ell_2\RA}=h_n(t_{\ell_1},\ldots,t_{\ell_2})$
for positive integers $\ell_1$ and $\ell_2$
($\ell_1 <\ell_2$) is given by
\label{pg:h_n}
$$
\prod_{i=\ell_1}^{\ell_2}\frac{1}{(1-z t_i)}
=\sum_{n\ge 0} h_n^{\LA\ell_1,\ell_2\RA} z^n,\quad
h_n^{\LA\ell_1,\ell_2\RA} =0 \ \mbox{for } n < 0.
$$

\begin{proposition} \label{prop:SchurC} 
 \cite[Theorem 4.5.1]{Sa}
Using the complete homogeneous symmetric functions $h_n:=h_n^{\LA1,g\RA}$,
we can express $\bold{s}_\Lambda$ by a ($g\times g$)
\textrm{Jacobi-Trudi Determinant}, $|a_{ij}|_{1\le i,j\le g}$ 
with $a_{ij}=h_{\Lambda_i+j-i}$:
$$
        \bold{s}_{\Lambda}(t) :=|h_{\Lambda_i+j-i}|,\quad
        h_n=
\left|\begin{array}{ccccc}
   T_1 & -1 & 0 & \cdots & \\
   2T_2 & T_1 & -2 & \cdots & \\
   \vdots & \vdots &\vdots&\ddots&\vdots\\
   (n-1)T_{n-1} & (n-2)T_{n-2} & (n-3)T_{n-3} & \cdots & 1-n\\
   nT_{n} & (n-1)T_{n-1} & (n-2)T_{n-2} & \cdots & T_1\\
\end{array} \right|,
$$
where $h_0 =1$, $h_{i<0} =0$ and $T_k:=
\frac{1}{k}\sum_{j=1}^{g} t_j^k$.
\end{proposition} 

We use the multi-index convention
 and define a map
for $\beta:= (\beta_1, \ldots, \beta_g)$,
 \label{pg:multi-index2}
$$
\wg(\beta):= 
((2g - N(0) - 1)\beta_1,
(2g - N(1) - 1)\beta_2,
 \ldots,
(2g - N(g-2) - 1)\beta_{g-1},
\beta_g) \in \ZZ^g.
$$
By regarding   
$S_{\Lambda}(T) :=\bs_{\Lambda}(t)$
as a function of $T$,
the following Proposition was stated by
Bukhshtaber,   Leykin and Enolskii
\cite{BEL} and proved by Nakayashiki \cite[Theorem 3]{N1}.

\begin{proposition} \label{prop:Nakayashiki}
The expansion of $\sigma(u)$ at the origin takes the form
$$
   \sigma(u) = S_{\Lambda}(T)|_{T_{\Lambda_i + g - i} = u_i} 
            + \sum_{|\wg(\alpha)|>|\Lambda| } c_\alpha u^\alpha
$$
where $c_\alpha\in \QQ[\lambda_j]$ and $S_\Lambda (T)$ is the lowest-order
term in the w-degree of the $u_i$; 
$\sigma(u)$ is homogeneous of degree  $|\Lambda|$ 
 with respect to the $\lambda$-degrees. 
\end{proposition}

Here we note that the factor $c$ in 
(\ref{def_sigma}) is fixed
so that there is no prefactor
in $S_{\Lambda}(T)|_{T_{\Lambda_i + g - i} = u_i}$. 
For the $(3,4)$ curve case, $\sigma$ is expanded 
around $\Theta^{2}$;
\begin{equation*}
        \sigma(u) = (u_1 - u_3 u_2^2 + \frac{1}{20} u_3^5) + 
        \mbox{higher weight terms}.
\label{eq:sigmaext}
\end{equation*}

\bigskip

\begin{remark}{\rm{
{\bf{Galois action on $\sigma$:}}
It should be noted that the action of the cyclic group $\mathbb{Z}_r$  on \(X\) 
induces its action on \(\JJ\), by the morphism \(\SSS^g(X)\to\JJ\), and 
on the coordinate space \(\mathbb{C}^g=\kappa^{-1}(\JJ)\) of
 the $\sigma$-function. 
Since the lattice \(\Pi\) is stable under this action of \(\mathbb{Z}_r\), 
we see that \(\sigma_{}(\hzeta_r u)\) also satisfies (\ref{eq:4.11}). 
We shall recall here that the space of the functions satisfying the relation (\ref{eq:4.11}) 
is one dimensional (Frobenius' theorem).  
Adding the leading terms of the expansion (\ref{prop:Nakayashiki}) 
and (\ref{eq:sigmaext}), we have arrived 
the formula  $\sigma_{}(\hzeta_r u)= \zeta_r^{a} \sigma_{}(u)$  
for some integer \(a\). 
Thus we will determine the exponent $a$. 
For $u_g^{(i)} \in \WW^1$ $(i = 1, 2, \ldots, g)$, 
and $t=(u_g^{(1)}, u_g^{(2)}, \ldots, u_g^{(g)})$,
$\sigma$ is expressed by
\begin{equation*}
   \sigma(u) = s_{\Lambda}(t) 
            + \sum_{|\wg(\alpha)|=|\Lambda|+1 } c_\alpha u^\alpha
            + \sum_{|\wg(\alpha)|=|\Lambda|+2 } c_\alpha u^\alpha
            + \cdots
\end{equation*}
and the action of $\hzeta$ on $\sigma(u)$ is given through the
terms with the same weight,
\begin{equation*}
   \sigma(\hzeta_r u) = \zeta_r^a s_{\Lambda}(t) 
            + \zeta_r^a \sum_{|\wg(\alpha)|=|\Lambda|+1 } c_\alpha u^\alpha
            + \zeta_r^a \sum_{|\wg(\alpha)|=|\Lambda|+2 } c_\alpha u^\alpha
            + \cdots.
\end{equation*}
On the other hand, since
{
$s_{\Lambda}(\hzeta_g u) =\zeta_r^{|\Lambda|(r_{g-1}+1)} s_{\Lambda}(u)$,
we have $a=|\Lambda| (r_{g-1}+1)|$:
\begin{equation}
\sigma(\hzeta_g u) =\zeta_r^{|\Lambda|(r_{g-1}+1)} \sigma(u).
\label{eq:sigam_zeta}
\end{equation}
}

}}
\end{remark}

\section{Vanishing of $\sigma$ function}

For the truncated Young diagram $\Lambda^{[k]}$ in
(\ref{eq:Lambdak})
we write the characteristics of the partition by
$(a_1, \ldots, a_{n_k}; b_1, \ldots, b_{n_k})$ and 
the cardinality of the set of pairs denoted by $n_k$
 \cite[\S 4.1, p. 51]{FH},
and define
$$
N_k := |\Lambda^{[k]}| = \sum_{i=1}^{n_k}(a_i+b_i+1).
$$
We note that $a_i$ and $b_i$ depend on $k$ of $\Lambda^{[k]}$.
In general, $a_i$ ($b_i$) of $\Lambda^{[k]}$  differs from 
$a_j$ ($b_j$) of $\Lambda^{[k]}$ if $i \neq j$.
Using them we have the quantities
$a_{n_k - i + 1} + b_{n_k - i + 1} + 1$, $i = 1, \ldots, n_k$.
Due to the arguments in \cite{MP2}, 
there exists an integer $j \in \{k+1, k+2, \ldots, g\}$ such
that
$a_i + b_i +1 =\Lambda_{j}+g-j$, for 
every $i = 1, \ldots, n_k$.
The correspondence is given by
$L^{[k]}(a_i, b_i):=j$.
For $i=n_k$ case, $a_{n_k} + b_{n_k} +1 =\Lambda_{k+1}+g-k-1$
and
$L^{[k]}(a_{n_k}, b_{n_k})=k+1$.

\begin{definition} \label{def:naturalk}
For $k = 1, 2, \ldots, g-1$,
and the characteristics of the partition of
$\Lambda^{[k]}$,  $(a_1, \ldots, a_{n_k}; b_1, \ldots, b_{n_k})$,
we define
$$
\natural_k :=\{
L^{[k]}(a_1,b_1), 
L^{[k]}(a_2,b_2), 
 \ldots,
L^{[k]}(a_{n_k}, b_{n_k})\},
$$ 
and 
$$
\natural_k^{(i)} := \left(\natural_k \setminus \{k + 1\}\right) \bigcup \{i\},
\quad \mbox{for} \quad i = 1, 2, \ldots, k.
$$
\end{definition}

From \cite[Corollary 5.4]{MP2}, we have
\begin{proposition} \label{prop:Fay}
For $u^{[k]} \in \Theta^k\setminus
\left(\Theta^k_1\cup \Theta^{k-1}\right)$,
$u^{[g]} \in \CC^g$,
$v\in \WW^{1}$, and $t \in \RR$ $(0<|t|<1)$, we have
\begin{enumerate}

\item
$\displaystyle{
        \left.\frac{\partial^{\ell}}
         {\partial {v_g}^{\ell}} \sigma(t v + u^{[k]})\right|_{v = 0} =
         0, \  \ell < N_k; \quad
        \left. \frac{\partial^{N_k}}
         {\partial {v_g}^{N_k}}
          \sigma(t v + u^{[k]})\right|_{v = 0} \neq 0,
}$ and

\item
$\displaystyle{
        \left.\frac{\partial^{\ell}}
         {\partial {u^{[g]}_g}^{\ell}} \sigma(u^{[g]})
         \right|_{u^{[g]} = u^{[k]}} = 0, \  \ell < N_k; \quad
        \left. \frac{\partial^{N_k}}
         {\partial {u^{[g]}_g}^{N_k}}
          \sigma(u^{[g]})\right|_{u^{[g]} = u^{[k]}} \neq 0.
}$

\end{enumerate}
\end{proposition}

Let $\mathcal{I}$ be the family of all finite 
sequences made up of unordered sets of numbers between 1 and $g$.
For an element $I_k \in \mathcal{I}$ and $u \in \mathbb{C}^g$, define\,{\rm{:}}
$$
        \sigma_{I_k}:=\left(
       \prod_{i \in I_k}
         \frac{\partial}{\partial u_i} \right) \sigma,
$$
$$
        \mwdeg(I_k) := \sum_{i \in I_k} \mwdeg(u_i).
$$

From \cite[Theorem 5.15]{MP2}, we have
\begin{proposition} \label{vanishingTh}
Let $\cI_g= \{ \emptyset \}$.
For each $k=1,2,\ldots,g$,
there exists a family of $\cI_k$ sequences consisting of 
$1, 2, \ldots, g$
whose element $I_{k}$ is such that
$\# I_{k} = n_k$, $\mwdeg(I_{k})\ge N_k$,
and as a function over 
$\kappa^{-1}(\Theta^k \setminus (\Theta^k_1 \cup \Theta^{k-1}))$,
\begin{equation}
        \sigma_{J_{k}}=\left\{
        \begin{array}{ll}
           \neq 0 & \mbox{ for } J_{k} =I_k\\
           = 0 & \mbox{ for } J_{k}\varsubsetneqq I_k. \\
        \end{array} \right.
\label{eq:Avn-nvn}
\end{equation}
Moreover, 
$\{ \natural_k, \natural_k^{(k)},
 \natural_k^{(k-1)}, \ldots, \natural_k^{(2)}, \natural_k^{(1)}\}\subset
\cI_k$ and $\mwdeg(\natural_{k})= N_k$.
\end{proposition}

\begin{remark} 
{\rm{
Here we have the convention that for $v \in \WW^k \subset \JJ$,
and $u \in \CC^g$,
$\sigma_{I_k}(v)$ means that it is given by 
$\displaystyle{\left(
       \prod_{i \in I_k}
         \frac{\partial}{\partial u_i} \right) \sigma(u) \Bigr|_{u = v}}$.
}}
\end{remark} 

For the case of an $(r,s)$ curve, 
the  action $\hzeta_r$ on $u$ via
its Abel preimage $(x_i, y_i) \to (x_i, \zeta_r y_i)$,
($\zeta_r^r = 1$, $\zeta_r\neq 1$), is given by:
{
\begin{equation*}
\sigma_{\natural_k}(\hzeta_r u)
= \zeta_r^{|\Lambda^{(k)}| (r_{g-1}+1)} \sigma_{\natural_k}(u).
\label{eq:zetak}
\end{equation*}
}

 A non-vanishing theorem for $\sigma$ over
a hyperelliptic $(2, 2g+1)$ curve $X$ of genus $g$ is
given in \cite{O1},
$$
 \natural_k:=\left\{
\begin{array}{cc} 
\{g, g - 2, \ldots, k + 3, k + 1\} & \mbox{if } g-k \mbox{ is odd}, \\
\{g - 1, g - 3, \ldots, k +3 , k +1 \} & \mbox{ otherwise}.
\end{array} \right.
$$ 
Similar results for $(3,4)$ and $(3,5)$ curves are given in \cite{O2,O3}.

We show some examples of $\natural_k$ in Table 1.1.
{\tiny{
\begin{equation*}
\centerline{
\vbox{
        \baselineskip =10pt
        \tabskip = 1em
        \halign{&\hfil#\hfil \cr
        \multispan7 \hfil Table 1.1 \hfil \cr
        \noalign{\smallskip}
        \noalign{\hrule height0.8pt}
        \noalign{\smallskip}
$(r,s)$ & \strut\vrule& $g$ & \strut\vrule & $\natural_1$ &$\natural_2$
 & $\natural_3$ &$ \natural_4$
& $\natural_5$ & $\natural_6$ & $\natural_7$ \cr
\noalign{\smallskip}
\noalign{\hrule height0.3pt}
\noalign{\smallskip}
$(2,3)$&\strut\vrule& $1$ & \strut\vrule \cr
$(2,5)$&\strut\vrule& $2$ & \strut\vrule & $\{2\}$\cr
$(2,7)$&\strut\vrule& $3$ & \strut\vrule & $\{2\}$&$\{3\}$\cr
$(2,9)$&\strut\vrule& $4$ & \strut\vrule & $\{2,4\}$&$\{3\}$&$\{4\}$ \cr
$(2,11)$&\strut\vrule& $5$ & \strut\vrule 
                    & $\{2,4\}$&$\{3,5\}$&$\{4\}$& $\{5\}$\cr
$(2,13)$&\strut\vrule& $6$ & \strut\vrule 
                    & $\{2,4,6\}$&$\{3,5\}$&$\{4,6\}$& $\{5\}$&$\{6\}$ \cr
$(2,15)$&\strut\vrule& $7$ &\strut\vrule 
                    & $\{2,4,6\}$&$\{3,5,7\}$&$\{4,6\}$& $\{5,7\}$&
                                           $\{6\}$&$\{7\}$\cr
$(2,17)$&\strut\vrule& $8$ & \strut\vrule 
                    & $\{2,4,6,8\}$&$\{3,5,7\}$&$\{4,6,8\}$& $\{5,7\}$&
                                 $\{6,8\}$&$\{7\}$& $\{8\}$ \cr
$(3,4)$&\strut\vrule& $3$ & \strut\vrule 
                    & $\{2\}$&$\{3\}$& & & & &  \cr
$(3,5)$&\strut\vrule& $4$ & \strut\vrule 
                    & $\{2\}$&$\{3\}$& $\{4\}$ & & & &  \cr
\noalign{\smallskip}
\noalign{\hrule height0.3pt}
\noalign{\smallskip}
\noalign{\hrule height0.8pt}
}
}
}
\end{equation*}
}}

In \cite{O2}, the partial derivative over multi-index $\natural_n$
was given by
$
        \sigma_{\natural_1}(u)=\frac{\partial^2}{\partial u_3^2} \sigma(u)
=\sigma_{33}(u)
$
but here we have instead defined
$
        \sigma_{\natural_1}(u)=\frac{\partial}{\partial u_2} \sigma(u)
=\sigma_{2}(u)$ because
$\sigma_{2}(u) =-\sigma_{33}(u)$.

Using them, Theorem 5.24 in \cite{MP2} gives the Jacobi inversion formulae
for a stratum in the Jacobian $\JJ$: 
\begin{theorem}\label{algebraic}
For $k < g$,
$(P_1, \ldots, P_k) \in \SSS^k(X\backslash \infty ) \setminus (\SSS^k_1(X)
\cap\SSS^k(X\backslash \infty ))$ and
$u = \pm w(P_1, \ldots, P_k)\in \kappa^{-1}(\Theta^k)$,
$$
\frac{ \sigma_{\natural_k^{(i)}}(u) }{\sigma_{\natural_k}(u)}
      = (-1)^{k-i+1} \mu_{k, i-1} (P_1, \ldots, P_k) .
$$
\end{theorem}

Since $\mu_{1,0}(P)=x$ for a certain point $P=(x,y)$, we have 
$$\sigma_{\natural_1^{(1)}}(u)/\sigma_{\natural_1}(u)=-x\qquad \mbox{for}\quad u =w(P).$$ 
For the case of a cyclic $(3,4)$ curve, we have $du_3 = dx/3y$
and thus, $d (\sigma_{\natural_1^{(1)}}(u)/\sigma_{\natural_1}(u))/d u_3 = -3y$
as mentioned in Introduction.

\section{Frobenius-Stickelberger relation}\label{CoordinateChange}
The prototype of such a formula was first found, for the elliptic case, 
in \cite{FS}. Generalisations of that result to general hyperelliptic curves,
and to examples of cyclic trigonal curves, have been given in 
\cite[Theorem 7.2]{O1}, \cite{O2} and \cite{O3}.
Here we recall these results, which will be needed below.

In the general hyperelliptic case we have:
\begin{proposition}\label{prop:FSH}
For the hyperelliptic $(2, 2g+1)$ curve and
a positive integer $n>1$,
let $(x_1,y_1), \cdots, (x_n, y_n)$  in  $X$, 
and $u^{(1)}, \cdots, u^{(n)}$  in  $\kappa^{-1}(\WW^1)$  be points such that  
$\kappa(u^{(i)})= \kappa w((x_i, y_i))$.
Then the following relation holds\,{\rm :}
\begin{equation}
\frac{\sigma_{\natural_n}(\sum_{i=1}^{n} u^{(i)})
\prod_{i<j}\sigma_{\natural_2}(u^{(i)} - u^{(j)}) }
{\prod_{i=1}^n\sigma_{\natural_1}(u^{(i)})^n}
=\epsilon_n \psi_n((x_1, y_1), \cdots, (x_n, y_n)),
\label{eq:FS0}
\end{equation}
where $\epsilon_n=(-1)^{g+n(n+1)/2}$ for $n\leq g$
and $\epsilon_n=(-1)^{(2n-g)(g-1)/2}$ for $n\geq g+1$.
\end{proposition}

In the case of a trigonal curve, from \cite[Theorem 5.3, Lemma 4.3]{O2}, we have the following Proposition
and Corollary:
\begin{proposition}\label{prop:FSRTC}{\rm{\cite[Theorem 5.3]{O2}}}
For trigonal  $(3, s)$ curve $(s=4,5)$, and
a positive integer $n>1$,
let $(x_1,y_1), \cdots, (x_n, y_n)$  in  $X$  
and  $u^{(1)}$, $\cdots$, $u^{(n)}$  in  $\kappa^{-1}(\WW^1)$ 
be points such that  $\kappa(u^{(i)})= \kappa w((x_i, y_i))$.
Then the following relation holds\,{\rm :}
\begin{equation}
\begin{array}{rl}
&\displaystyle{\frac{\sigma_{\natural_n}(\sum_{i=1}^{n} u^{(i)})
\prod_{i<j}\sigma_{\natural_2}(u^{(i)} +\hzeta_3 u^{(j)}) 
\sigma_{\natural_2}(u^{(i)} +\hzeta_3^2 u^{(j)}) }
{\prod_{i=1}^n\sigma_{\natural_1}(u^{(i)})^{2n-1}}}\\
&=\psi_n((x_1, y_1), \cdots, (x_n, y_n)) 
\varphi_n((x_1, y_1), \cdots, (x_n, y_n)).
\label{eq:FS3}
\end{array}
\end{equation}
\end{proposition}

\begin{corollary}\label{cor:Lim342}{\rm{\cite[Theorem 5.3, Lemma 4.3]{O2}}}
For a cyclic $(3,4)$ curve, in particular,
let $(x, y)$ and 
 $(x_i,y_i)_{i=1,2}$ be points in $X\times \SSS^2(X)$
and $(u, v_1, v_2)$ be a point in
 $\kappa^{-1} \WW^1 
\times \kappa^{-1} \WW^1 \times \kappa^{-1} \WW^1$
such that
$u = w(x,y)$ and $v_i = w(x_i,y_i)$.
Then the following three relations hold:
{
\begin{equation*}
\begin{array}{rl}
&\displaystyle{\frac{\sigma_{\natural_{3}}(u + v_1 + v_2)
\prod_{a=1}^2\left(
 \sigma_{\natural_{2}}(u +\hzeta^a v_1) 
 \sigma_{\natural_{2}}(u +\hzeta^a v_2)
 \sigma_{\natural_{2}}(v_1 +\hzeta^a v_2)\right)
 }
{\sigma_{\natural_1}(u)^5 \sigma_{\natural_1}(v_1)^5
\sigma_{\natural_1}(v_2)^5 }} \\
&=- (x-x_1)(x-x_2)(x_1-x_2)\left(
y(x_1-x_2)  -y_1(x-x_2)  +y_2(x-x_1)\right) ,
\end{array}
\end{equation*}
\begin{equation*}
\frac{\sigma_{\natural_{2}}(v_1 + v_2)
 \sigma_{\natural_{2}}(v_1 +\hzeta v_2) 
 \sigma_{\natural_{2}}(v_1 +\hzeta^2 v_2)
 }
{\sigma_{\natural_1}(v_1)^3
\sigma_{\natural_1}(v_2)^3 } 
= \left|
\begin{array}{cc}
1 & x_{1} \\
1 & x_{2}  \\
\end{array}
\right|^2 ,
\end{equation*}
}
$$
\frac{\sigma_{\natural_{2}}(2u)}{\sigma_{\natural_1}(u)^4} = 3 y^2.
$$
\end{corollary}

\section{The Prime form of an $(r,s)$ curve and the sigma function}

\subsection{The Prime form of an $(r,s)$ curve}

For the universal covering $\tX$ of $X$,
let us define the prime form as follows:
\begin{definition}
The prime form $E(P, Q)$ for 
a point $P, Q \in \tX$ is
defined by
\begin{equation}
E(P, Q) := \frac{\theta
     \left[\begin{array}{c} \Edelta'\\ \Edelta''\end{array}\right]
               (\int^P_Q \hat \nuI)}{
               \sqrt{\zeta(P)} \sqrt{\zeta(Q)}},
\label{eq:EPQ}
\end{equation}
where
$\Edelta'$ and $\Edelta''$ are a fixed nonsingular odd theta
characteristics and $\zeta(P)$ is
$$
\zeta(P) := \sum_{i=1}^g
\left(\partial_i\theta\left[\begin{array}{c} \Edelta'\\ \Edelta''\end{array}
          \right] (0)\right) \hat \nuI_i(P).
$$
\end{definition}

We should note that the prime form is a $(-1/2)\otimes(-1/2)$ form
of $\tX$.
Due to the Abel map, for  $\hell \in \hPi$, we have
a natural translation action $\gamma_\hell$ on $\tX$ such that
$\hw(\gamma_\hell(P)) = \hw(P) + \hell$.

In "Tata lectures on theta II" \cite[3.210]{Mu},
 Mumford shows that
\begin{proposition} \label{prop:EPQ}
The prime form satisfies the following relations:
\begin{enumerate}
\item 
$E(P_1, P_2)$ vanishes iff $P_1 = P_2$.

{
\item $E(P_1, P_2)$ has a first order zero along the diagonal
$\Delta \subset X \times X$.
}

{
\item $E(P_1, P_2) = -E(P_2, P_1)$.
}

\item 
Let $t_i$ be a local parameter at $P_i$ of $X$ 
such that $\zeta(P_i) = dt_i$, $(i=1,2)$, 
$$
E(P_1, P_2) = \frac{t_1 - t_2}{\sqrt{dt_1}\sqrt{dt_2}}
(1 + d_{\ge}\left((t_1 - t_2)^2\right).
$$

\item $E(P,Q)$ has the properties for the action $\gamma_\hell$
for $\hell \in \hPi$,
$$
E(\gamma_\hell P, Q)
= \ee^{-2\pi\ii( \frac{1}{2} ^t\ell^{\prime\prime}\tau \ell''
+^t\ell'' z +^t \ell'\Edelta'' -^t \ell''\Edelta')} 
 E(P, Q).
$$
\end{enumerate}
\end{proposition}

\begin{proposition}
$$
\frac{E(P,Q)}{ E(P,Q')} = \exp \int^P \tau_{Q, Q'}.
$$
where $\tau_{Q, Q'}$ is the differential of the third kind,
which has residues $+1,-1$ at ${Q_1, Q_2}$,
is regular everywhere else,  and is  normalized 
{\it{i.e.}}, $\int_{\alpha_i} \tau_{P, Q} = 0$.
\end{proposition}

\begin{remark}
This fact may be used to construct the Green's function for the Laplacian on the Riemann surface $X$.
\end{remark}

Following Nakayashiki's definition of the prime form \cite[Definition 3]{N1}
\begin{definition}
we define another prime form:
\begin{equation}
\cE(P,Q):= \ee^{^t(w(P)-w(Q))\gamma (w(P)-w(Q)))} E(P,Q) 
\label{eq:cEPQ}
\end{equation}
where $\gamma := \omega^{\prime -1} \eta'$.
\end{definition}
\begin{remark} 
{\rm{
Our definition differs slightly from that in \cite{N1}, as
here $\cE(P,Q)$, like Fay's $E(P,Q)$, is defined as a $(-1/2,-1/2)$ form.
}}
\end{remark}

\begin{proposition}\label{prop:cEPQ}
\begin{enumerate}
\item 
$\cE(P_1, P_2)$ vanishes iff $P_1 = P_2$.

{
\item $\cE(P_1, P_2)$ has a first order zero along the diagonal
$\Delta \subset X \times X$.
}

{
\item $\cE(P_1, P_2) = -\cE(P_2, P_1)$.
}

\item 
Let $t_i$ be a local parameter at $P_i$ of $X$ 
such that $\zeta(P_i) = dt_i$, $(i=1,2)$ 
$$
\cE(P_1, P_2) = \frac{t_1 - t_2}{\sqrt{dt_1}\sqrt{dt_2}}
(1 + d_{\ge}\left((t_1 - t_2)^2\right).
$$

\item $\cE(P,Q)$ has the properties for the action $\gamma_\ell$
on $\tX$,
$$
\cE(\gamma_\ell P, Q):= 
\chi(\ell) L\left( (w(P)-w(Q)) + \frac{1}{2} (\ell, \ell)\right) \cE(P,Q) .
$$
where $w(\gamma_\ell(P)) = w(P) + \ell$.
\end{enumerate}
\end{proposition}

{
\begin{proof}
(i)-(iv) are obvious from the previous Proposition.

From \cite[Proposition 8]{N1}, we have (v).
\end{proof}
}

\begin{proposition} \label{prop:7.6}
If a $(-1/2)\otimes(-1/2)$-form on $\tX \times \tX$ has the properties
(i), (ii), (iii), (iv) and (v)
in Proposition \ref{prop:cEPQ}, it has the form of (\ref{eq:cEPQ}).
\end{proposition}

\begin{proof}
If there is another $(-1/2)\otimes(-1/2)$-form $\cE'(P,Q)$
on $\tX \times \tX$ with these properties,
the quasi-periodic properties (iv) show that
$\cE'(P,Q)/\cE(P,Q)$ must be a meromorphic function
over $\CC^g$. Further the vanishing properties (iii) determine
that the meromorphic function has no zero and thus must be
a constant function. By (iv), this constant is equal to $1$.
\end{proof}

This uniqueness theorem will enable us to prove the main results of 
this paper, the following explicit formulae for the Prime form.

For any hyperelliptic $(2,2g+1)$ curve, we have
$$
\cE(P,Q) =\frac{ \sigma_{\natural_{2}}(u -  v)}
{\sqrt{du_1}\sqrt{d v_1}}.
$$

Further, for a cyclic $(3,4)$ curve, we have
$$\cE(P,Q) =\frac{ \sigma(u - v)}
{\sqrt{du_1}\sqrt{d v_1}}
\equiv
\frac{ \sigma_{\natural_{3}}(u - v)}
{\sqrt{du_1}\sqrt{d v_1}}.
$$

We will prove these results in the next two subsections.

\subsection{The Prime form - hyperelliptic case}

We will prove the theorem for the case of a  hyperelliptic curve;
first we will establish several preliminary results.

First, from Proposition \ref{prop:FSH}, we directly have the
following;
\begin{proposition} \label{lm:add2H}
Let $(x, y)$ and $(x',y')$ be points in $X$
and 
$(u, v) \in $ $\CC^g \times \CC^g$ such that
$u = w(x,y)$ and
$v = w(x',y')$.
$$
\frac{
\sigma_{\natural_2}(u+ v)
\sigma_{\natural_2}(u- v)}
{\sigma_{\natural_1}(u)^2\sigma_{\natural_1}(v)^2}
= (x - x').
$$
\end{proposition}

\begin{remark} 
{\rm{
Proposition \ref{lm:add2H} 
is a natural generalisation of
the addition formula in the $g=1$ case (\ref{eq:add_g1}).
}}
\end{remark} 

Let $(x, y)$ and $(x',y')$ be points in $X$
and $(u, v) \in $ $\CC^g \times \CC^g$ such that
$u = w(x,y)$ and $v = w(x',y')$.

From \cite[Lemma 9.1]{O1}, we have the relation:
\begin{equation*}
\lim_{u \to v}\frac{\sigma_{\natural_{2}}(u - v)
 } {u_1-v_1}
= 1.
\end{equation*}

This implies the following relation:
\begin{lemma}\label{simplezero}
Let $(x, y)$ and $(x',y')$ be points in $X$
and $(u, v) \in $ $\CC^g \times \CC^g$ such that
$u = w(x,y)$ and $v = w(x',y')$.
$$
\frac{ \sigma_{\natural_{2}}(u -  v)}{\sqrt{du_1}\sqrt{d v_1}}
 = \frac{u_1- v_1}{\sqrt{d u_1}\sqrt{d v_1}} + \cdots.
$$
\end{lemma}
\begin{remark}
{\rm{
 This establishes 
conditions {\it{(ii)}} and {\it{(iv)}}
 of \ref{prop:7.6}.
Further Lemma 4.3 in \cite{O1} shows that
$\sigma$ function is an odd or even function with respect to 
$u\in \CC^g$. Hence this also establishes 
{\it{(iii)}} of \ref{prop:7.6}.
}}
\end{remark}

Secondly, from \cite[Proposition 7.5]{O1}, we have the relation:
\begin{lemma}\label{lm:LimHyp}{\rm{\cite[Proposition 7.5]{O1}}}
Let $(x, y)$ and $(x',y')$ be points in $X$
and $(u, v) \in $ $\CC^g \times \CC^g$ such that
$u = w(x,y)$, $v = w(x',y')$  and  $\ell$
($=2\omega'\ell'+2\omega''\ell''$) $\in\Pi$.
Then the following quasi-periodicity relation holds
\begin{equation*}
\frac{ \sigma_{\natural_{2}}(u + \ell -  v)}{\sqrt{du_1}\sqrt{d v_1}}
=
\exp(L(u+\frac{1}{2}\ell, \ell)) \chi(\ell)
\frac{ \sigma_{\natural_{2}}(u + \ell -  v)}{\sqrt{du_1}\sqrt{d v_1}}.
\end{equation*}
\end{lemma}
\begin{remark}
{\rm{
This establishes condition {\em{(v)}} of \ref{prop:7.6}.
}}
\end{remark}

\begin{lemma}\label{lm:LimHyp}
Let $(x, y)$ and $(x',y')$ be points in $X$
and $(u, v) \in $ 
$\kappa^{-1}\left(\Theta^1\setminus \Theta^{0}\right)$ 
$\times \CC^g$ such that
$u = w(x,y)$ and $v = w(x',y')$.
Then the following relation holds
\begin{equation*}
\lim_{ v \to 0}
\frac{\sigma_{\natural_{2}}(u-v)}{ \sqrt{d u_1} \sqrt{d v_1}} \neq 0.
\end{equation*}
\end{lemma}

\begin{proof}
From \cite[Proposition 5.26]{MP2} and Proposition \ref{prop:Fay},
$\sigma_{\natural_{2}}(u-v)
=c_1\sigma_{g^{N_2}}(u-v)
=c_2\sigma_{g^{N_1}}(u)v_g^{g-1} + d_{>}(v_g^g)$, where
$c_1$ and $c_2$ are constants.
Since  the weight of $v_1$ is $ 2g-1$, we have the relation.
\end{proof}

\begin{remark}
{\rm{
Thus $\frac{\sigma_{\natural_{2}}(u-v)}{ \sqrt{d u_1} \sqrt{d v_1}}$ 
has zeroes only where $(x,y)=(x',y')$, and by (\ref{simplezero}) 
these are simple, giving condition {\em{(i)}} of \ref{prop:7.6}.
}}
\end{remark}

Assembling these results, 

 $\cE(P, Q)$ is equal to 
$\displaystyle{\frac{ \sigma_{\natural_{2}}(u -  v)}
{\sqrt{du_1}\sqrt{d v_1}}}$
due to Proposition \ref{prop:7.6}.
Thus we have proved the following:

\begin{theorem}\label{th:Ehyp}
Let $(x, y)$ and $(x',y')$ be points in $X$
and 
$(u, v) \in $ $\CC^g \times \CC^g$ such that
$u = w(x,y)$ and
$v = w(x',y')$.
Then the following relation holds
$$
\cE(P,Q) =\frac{ \sigma_{\natural_{2}}(u -  v)}
{\sqrt{du_1}\sqrt{d v_1}}.
$$
Here $d u_1 := \nuI_1(x,y)$ and
$d v_1 := \nuI_1(x',y')$.
\end{theorem}. 

\begin{remark}
{\rm{
As in \cite[Proposition 4.9]{KMP},
by letting 
$(x_i,y_i) \in X$  $(i=1, \ldots, g)$, $(x'_j,y'_j)\in X$ $(j=1, 2)$,
 $u\in \CC^g$,
$v := v^{[1]} +v^{[2]}\in\kappa^{-1}(\WW^2)$, and
$v^{[j]}\in\kappa^{-1}(\WW^1)$ $(j=1,2)$ be points
such that  $u=w((x_1, y_1),\ldots,(x_g,y_g))$
and  $v^{[j]}=w((x_j', y_j'))$, $(j=1,2)$,
 the following relation holds\,{\rm :} 
\begin{equation}
\frac{\sigma(u + v) \sigma(u - v) }
{\sigma(u)^2 \sigma_{\natural_2}(v)^2}
= \frac{f(x_1', x_2')-2 y_1' y_2' } {(x'_1-x'_2)^2} -
\sum_{i=1}^g \sum_{j=1}^g \wp_{i j}(u)
 {x'_1}^{i-1} {x'_2}^{j-1}.
\label{eq:prop:fay}
\end{equation}
This corresponds to Fay's formula \cite[(39)]{F} using the
prime form, which is the basis of ``Fay's trisecant identity''.
The relation is also clearly related to the two-dimensional Toda lattice
equation \cite{KMP}.
}}
\end{remark}


\subsection{Prime form - a cyclic trigonal curve}

Again, before proving our main theorem, we state some preliminary results.

First we consider the action of the Galois group on the $\sigma$-function
and its derivatives 
due to \cite[Lemma 4.1, 4.2]{O2} and 
\cite[Lemma 4.2]{EEMOP}:
\begin{lemma}
For $u\in \CC^3$ 
and action $\hzeta_3$ on $u$ via
its Abel preimage $(x_i, y_i) \to (x_i, \zeta_3 y_i)$,
corresponding to Remark \ref{rmk:sigma_zeta},
we have 
\begin{equation}
\sigma_{}(\hzeta_3 u)= \zeta_3 \sigma_{}(u) \equiv 
\zeta_3 \sigma_{\natural_3}(u), \quad
\sigma_{\natural_2}(\hzeta_3 u)= \zeta_3^2 \sigma_{\natural_2}(u),
\quad
\sigma_{\natural_1}(\hzeta_3 u)= \sigma_{\natural_1}(u),
\label{eq:sign2}
\end{equation}
and $\sigma(-u) = - \sigma(u)$.
\end{lemma}

Secondly, from \cite[Lemma 6.1]{O2}, we have
\begin{lemma}
Let $(x, y)$ and $(x',y')$ be points in $X$
and $(u, v) \in $ $\CC^3 \times \CC^3$ such that
$u = w(x,y)$ and
$v = w(x',y')$.
Then the following relations hold
\begin{equation*}
\lim_{u \to v}
\frac{\sigma_{\natural_1}(u)(x-x')}
{\sigma_{\natural_{2}}(u +\hzeta_3 v) } = (\zeta_3-1) y,
\quad
\lim_{u \to v}
\frac{\sigma_{\natural_1}(u)(x-x')}
{\sigma_{\natural_{2}}(u +\hzeta_3^2 v) } = (\zeta_3^2-1) y.
\end{equation*}
\end{lemma}

\begin{proof}
Noting $\sigma_{\natural_2} = \sigma_3$
and $\sigma_{\natural_2^{(2)}} = \sigma_2 =\sigma_{\natural_1}$,
from Theorem 5.1 in \cite{MP1} and
Theorem 5.23 in \cite{MP2},
we have the relation,
$$
\frac{\sigma_{\natural_1}(u +  v)}{\sigma_{\natural_2}(u +  v)}
= (-1)^{3-2+1}\frac{y -y'}{x -x'}.
$$
Thus
\begin{equation*}
\begin{array}{rl}
\displaystyle{
\lim_{u \to v}
 \frac{x-x'}{\sigma_{\natural_2}(u + \hzeta_3 v)}
}
&=
\displaystyle{
\lim_{u \to v}
 \frac{y-\zeta_3 y'}{\sigma_{\natural_1}(u + \hzeta_3 v)}
}\\
&= 
\displaystyle{
\frac{(1-\zeta_3) y'}{\sigma_{\natural_1}((1 + \hzeta_3) v)}
= \frac{(1-\zeta_3) y'}{\sigma_{\natural_1}(- \hzeta_3^2 v)}
= -\frac{(1-\zeta_3) y'}{\sigma_{\natural_1}(v)}.
}
\end{array}
\end{equation*}
The stated relations then follow.
\end{proof}

Next, from the Frobenius-Stickelberger relations, we have 
the following Proposition:
\begin{proposition} \label{prop:add2T}
Let $(x, y)$ and $(x',y')$ be points in $X$
and $(u, v) \in $ $\CC^3 \times \CC^3$ such that
$u = w(x,y)$ and
$v = w(x',y')$.
Then the following relation holds
\begin{equation*}
\frac{
\sigma_{}(u - v)
 \sigma_{\natural_{2}}(u + v) 
 }
{\sigma_{\natural_1}(u)^2
\sigma_{\natural_1}(v)^{2} } 
= \frac{
\sigma_{\natural_{3}}(u + \hzeta_3 v + \hzeta_3^2 v)
 \sigma_{\natural_{2}}(u + v) 
 }
{\sigma_{\natural_1}(u)^2
\sigma_{\natural_1}(v)^{2} } 
= (x-x').
\end{equation*}
\end{proposition}

\begin{proof}
By letting $v^{(a)} = w(x_a, y_a)$ $(a=1,2)$,
the left hand side is equal to
\begin{equation*}
\begin{array}{rl}
&
\displaystyle{
\lim_{v_{(a)} \to \zeta_3^a v}
\epsilon 
\psi_{3}((x,y),(x_1,y_1),(x_2,y_2))
\varphi_{3}((x,y),(x_1,y_1),(x_2,y_2)) 
}
\\
&
\displaystyle{
\quad \times
\frac{\sigma_{\natural_1}(u)^3\sigma_{\natural_1}(v)^{3}}
{\left(\prod_{a=0}^2
 \sigma_{\natural_{2}}(u +\hzeta_3^a v)\right)}
\frac{
\sigma_{\natural_{1}}(v^{(1)})^2
\sigma_{\natural_{1}}(v^{(2)})^2}
{\sigma_{\natural_{2}}(v^{(1)}+\hzeta_3^2 v^{(2)})}
\frac{\sigma_{\natural_{1}}(v^{(1)})}
{
\sigma_{\natural_{2}}(\hzeta_3 v^{(1)} + \hzeta_3^2 v^{(2)})
}
}
\\
&=
\displaystyle{
(x-x')^3(\zeta_3-\zeta_3^2)y'
\lim_{v_a \to \zeta_3^a v}
(x_1-x_2)\cdot \frac{1}{(x-x')^2}\frac{1}{3(y')^2}
\frac{\sigma_{\natural_{1}}(v^{(1)})}
{
\zeta_3^2\sigma_{\natural_{2}}(v^{(1)} + \hzeta_3 v^{(2)})}
}\\
&=
\displaystyle{
(\zeta_3-\zeta_3^2)(x-x')\frac{\zeta_3-1}{3 \zeta_3^2}
=(x-x').
}
\end{array}
\end{equation*}
\end{proof}
  
\begin{remark} 
{\rm{
Proposition \ref{prop:add2T} is the trigonal analogue of 
Proposition \ref{lm:add2H}, and similarly
generalises the addition formula in the $g=1$ case (\ref{eq:add_g1}).
}}
\end{remark}

We now look at the order of vanishing of 
$\sigma_{\natural_3} \equiv\sigma$:
\begin{lemma}\label{lm:Lim34}
Let $(x, y)$ and $(x',y')$ be points in $X$
and $(u, v) \in $ $\CC^3 \times \CC^3$ such that
$u = w(x,y)$ and
$v = w(x',y')$.
Then the following relation holds
\begin{equation*}
\lim_{u \to v}\frac{\sigma_{\natural_{3}}(u + \hzeta_3 v + \hzeta_3^2 v)
 } {u_1-v_1}
=\lim_{u \to v}\frac{\sigma_{\natural_{3}}(u - v)
 } {u_1-v_1}
= 1.
\end{equation*}
\end{lemma}

\begin{proof}
Noting $\displaystyle{
\lim_{u \to v} \frac{ (x-x')}{u_1-v_1} = \frac{d }{d v_1} x(v) = 3 y(v)^2}$,
we have
\begin{equation*}
\lim_{u \to v}\frac{\sigma_{\natural_{3}}(u + \hzeta_3 v + \hzeta_3^2 v)
 } {u_1-v_1}
=
\lim_{u \to v}
\frac{ \sigma_{\natural_1}(u)^2 \sigma_{\natural_1}(v)^2 }
{ \sigma_{\natural_2}(2v) }
\frac{ (x-x')}{u_1-v_1}  
= 1.
\end{equation*}
\end{proof}

Hence we see that:
\begin{lemma}\label{lm:Lim34}
Let $(x, y)$ and $(x',y')$ be points in $X$
and $(u, v) \in $ 
$\kappa^{-1}\left( \Theta^1\setminus \Theta^{0}\right)$ $\times \CC^3$ such that
$u = w(x,y)$, and $v = w(x',y')$.
Then the following relation holds
\begin{equation*}
\lim_{ v \to 0}
\frac{\sigma_{\natural_{3}}(u-v)}{ \sqrt{d u_1} \sqrt{d v_1}} \neq 0.
\end{equation*}
\end{lemma}

\begin{proof}
Around $v=0$, we have $\sigma(u-v) = \frac{1}{20} \sigma_{33}(u) v_3^2 
+ \cdots$ whereas $dv_1 = v_3^4 d v_3 + \cdots$.
\end{proof}
\bigskip

Thus, again using Proposition \ref{prop:7.6} we have proved the following:
\begin{theorem} \label{thm:tri}
For the cyclic trigonal $(3,4)$ curve, 
$$
\cE(P,Q) =
\frac{ \sigma(u 
+ \zeta_3 v +\zeta_3^2 v)}
{\sqrt{-3}\sqrt{du_1}\sqrt{d v_1}}
=
\frac{ \sigma(u - v)}
{\sqrt{du_1}\sqrt{d v_1}}
\equiv
\frac{ \sigma_{\natural_{3}}(u - v)}
{\sqrt{du_1}\sqrt{d v_1}}.
$$
\end{theorem}

\bigskip
\section{Applications}
\subsection{Benney reductions}
Part of this work was motivated by the problem of simply describing certain
reductions of the Benney hierarchy. As described in \cite{G}, this problem
reduces geometrically to the construction of Schwartz-Christoffel mappings
from the upper half $p$-plane to a domain consisting of the upper half $\lambda$-plane,
minus finitely many straight slits, running from fixed base points on the
real axis, to variable end points; the slits should make internal angles
which are integer multiples of $\pi/r$ with one another and with the real
axis. The slits are images of disjoint segments of the real $p$-axis. 
In this case the integral formula for the Schwartz-Christoffel map
is algebraic; it is a second kind Abelian integral on a cyclic $(r,s)$ curve.
The domain and its image are illustrated schematically in the figures below,
where we choose $(r,s)=(2,5)$ for simplicity - this gives the problem of
evaluating a hyperelliptic integral of genus 2. In this case there are three
slits $\gamma_i$, perpendicular to the real axis. The respective end points
of these slits $\hat{\lambda}_i$ are the Riemann invariants of the  reduced
Benney system characterised by this family of mappings. 
The points $\hat{p}_i$ are the characteristic velocities of these 
invariants for the lowest non-trivial flow of the hierarchy.

\begin{center}
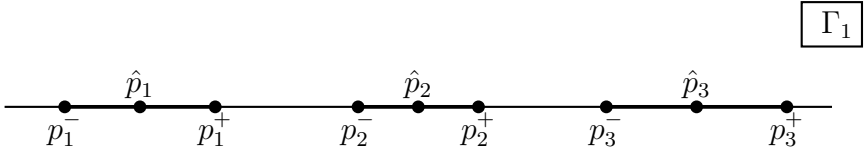
\begin{figure}[ht]
\begin{pspicture}(1,-0.5)(13,1.5) 
\psline(1.5,0)(12.5,0)
\psline[linewidth=0.5mm]{*-*}(2.3,0)(3.3,0)
\psline[linewidth=0.5mm]{-*}(3.3,0)(4.3,0)
\psline[linewidth=0.5mm]{*-*}(6.2,0)(7,0)
\psline[linewidth=0.5mm]{-*}(7,0)(7.8,0)
\psline[linewidth=0.5mm]{*-*}(9.5,0)(10.7,0)
\psline[linewidth=0.5mm]{-*}(10.7,0)(11.9,0)

\rput(2.3,-0.3){$p_1^-$}
\rput(3.3, 0.3){$\hat{p}_1$}
\rput(4.3,-0.3){$p_1^+$}
\rput(6.2,-0.3){$p_2^-$}
\rput(7, 0.3){$\hat{p}_2$}
\rput(7.8,-0.3){$p_2^+$}
\rput(9.5,-0.3){$p_3^-$}
\rput(10.7, 0.3){$\hat{p}_3$}
\rput(11.9,-0.3){$p_3^+$}
\rput(12.5,1.1){\psframebox{ $ \Gamma_1$}}
\end{pspicture} 
\caption{The $p$-plane for the genus 2 hyperelliptic reduction.} \label{fig:phyper}
\end{figure}

\begin{figure}[ht]
\begin{pspicture}(0,-0.5)(12,4.5) 
\psline(0,0)(12,0)
\psline{-*}(2.3,0)(2.3,2.5)
\psline{-*}(6,0)(6,1.5)
\psline{-*}(9.7,0)(9.7,3.5)
\psline[linewidth=0.5mm]{-*}(0,0.2)(2.1,0.2)
\psline[linewidth=0.5mm](2.1,0.2)(2.1,2.5)
\pscurve[linewidth=0.5mm](2.1,2.5)(2.3,2.7)(2.5,2.5)
\psline[linewidth=0.5mm]{-*}(2.5,2.5)(2.5,0.2)
\psline[linewidth=0.5mm]{-*}(2.5,0.2)(5.8,0.2)
\psline[linewidth=0.5mm](5.8,0.2)(5.8,1.5)
\pscurve[linewidth=0.5mm](5.8,1.5)(6,1.7)(6.2,1.5)
\psline[linewidth=0.5mm]{-*}(6.2,1.5)(6.2,0.2)
\psline[linewidth=0.5mm]{-*}(6.2,0.2)(9.5,0.2)
\psline[linewidth=0.5mm](9.5,0.2)(9.5,3.5)
\pscurve[linewidth=0.5mm](9.5,3.5)(9.7,3.7)(9.9,3.5)
\psline[linewidth=0.5mm]{-*}(9.9,3.5)(9.9,0.2)
\psline[linewidth=0.5mm](9.9,0.2)(12,0.2)
\rput(2.3,-0.3){$\lambda_1^0$}
\rput(6,-0.3){$\lambda_2^0$}
\rput(9.7,-0.3){$\lambda_3^0$}
\rput(2.95,3){ $\hat{\lambda}_1= \lambda( \hat{p}_1)$}
\rput(6.65,2){ $\hat{\lambda}_2= \lambda( \hat{p}_2)$}
\rput(10.35,4){ $\hat{\lambda}_3= \lambda( \hat{p}_3)$}
\rput(1.5,0.5){ $\lambda( p_1^-)$}
\rput(3,0.5){ $\lambda( p_1^+)$}
\rput(5.2,0.5){ $\lambda( p_2^-)$}
\rput(6.7,0.5){ $\lambda( p_2^+)$}
\rput(8.9,0.5){ $\lambda( p_3^-)$}
\rput(10.4,0.5){ $\lambda( p_3^+)$}
\rput(1.8,2.0){ $ \gamma_1$}
\rput(5.5,1.2){ $ \gamma_2$}
\rput(9.2,2.9){ $ \gamma_3$}
\rput(11.5,2.3){\psframebox{ $ \Gamma_2$}}
\end{pspicture}
\caption{The $\lambda$-plane associated with figure~\ref{fig:phyper}.}
\end{figure}
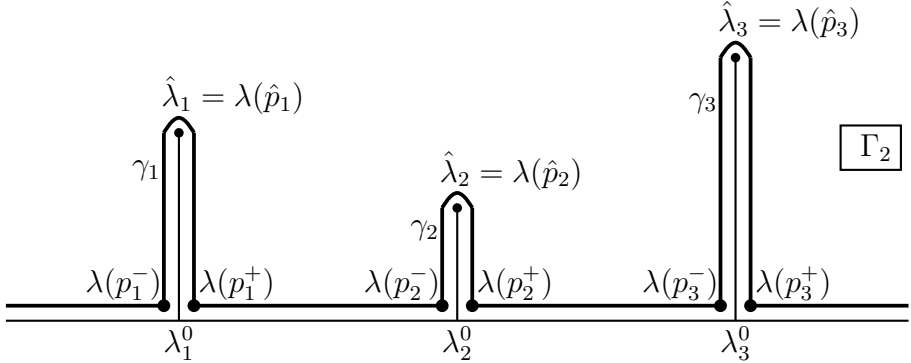 

\end{center}
The problem, in all these cases, reduces to finding a meromorphic function
with $r$ simple poles, as $p\rightarrow \infty$ on each sheet. If the base
point of the Abel map is taken as one of the branch points, then the Abel
images of the point at infinity  on the $n$-th sheet will be some point $\hat{\zeta}^n v$; here $v$ will depend on the moduli of the curve.
In all known cases, the mapping splits into a sum of two terms, one holomorphic,
while the other meromorphic term is written in terms of logarithmic 
derivatives of some derivative of $\sigma$.

In \cite{YG} this mapping was worked out in the elliptic case,
giving the formula
$$\lambda =k( \frac{d}{d u}\ln( \sigma(u-v)\sigma(u+v))) - k u+c,$$
where $k,c$ are constants. 

Further, in \cite{BG1} the analogous map was constructed explicitly for 
the case of a 
hyperelliptic curve of genus two, giving the result
$$\lambda \simeq k  \frac{d}{d u_1}\ln(  \sigma(u-v)\sigma(u+v)  ),$$
up to holomorphic terms linear in $u$.

In \cite{BG2} this approach was extended to higher genus hyperelliptic curves.
That result was not in the same form as those for lower genus; further for
genus higher than three the resulting formulae make sense only as limits.
A more systematic approach, using the expression for the Prime form given
above, yields in all these cases an analogous formula:
$$\lambda \simeq k( \frac{\partial}{\partial u_1} 
\ln(\sigma_{\natural_2}(u-u_0)\sigma_{\natural_2}(u+u_0))),$$
again up to holomorphic terms. This formula clearly reduces to the above
two in those  cases, consistent with the result given by Crowdy in \cite{C1},
\cite{C2}. There he had used the Schottky-Klein prime function, 
which he was able to
compute numerically for certain hyperelliptic curves.

There are similar problems with general non-hyperelliptic curves, 
two particular examples
of which were studied in \cite{BG3} and \cite{EG}. 
In the latter case, it was noted that not only $\sigma$, 
but all its first derivatives, vanished for  $u \in \WW^1$, a 
phenomenon which also is found for higher-order hyperelliptic curves.
The expression for the reduction given in \cite{EG} was thus given in 
terms of a logarithmic derivative of a second derivative of $\sigma$.

The methods used in \cite{MP1, MP2}, and above, 
are considerably more powerful; as these have enabled us to show a  
simple relationship 
between the Prime form and $\sigma$, finding a general expression for 
such reductions should now be comparatively straightforward.
The analogous formulae to those given above, in these cases, 
are expected to reduce to a sum of $r$ logarithmic derivatives, each having
a simple pole on one sheet, as $p\rightarrow \infty$. 
The approach given in \cite{G} suggests the appropriate
formula for this mapping in this case is, up to holomorphic terms,
$$ \lambda \simeq k \sum_{i=0}^{r-1} \zeta^i \frac{\partial}{\partial v_1}\ln(\cE(u-
\hat{\zeta}^{-i} v)),$$
where $k$ is a known constant. 

For the cyclic $(3,4)$ case, this lets us write down:
$$\lambda \simeq k \sum_{i=0}^{2} \zeta^i \frac{\partial}{\partial v_1}\ln(\sigma(u-
\hat{\zeta}^{-i} v)),$$
where $k$ is a known constant.
This should be contrasted with an analogous expression for the momentum $p$ in terms of $u$
$$p \simeq k \sum_{i=0}^{2}  \frac{\partial}{\partial v_1}\ln(\sigma(u-
\hat{\zeta}^{-i} v)).$$
The mapping from $p$ to $\lambda$ thus satisfies
$$\lambda \simeq \zeta^i p +O(1)\qquad 
\mbox{as} u\rightarrow \hat{\zeta}^{i} v.$$ 

\subsection{Further applications}
Other conformal mapping problems might potentially also become more 
tractable using this approach. 
Particular examples include mathematical problems such as the 
construction of multiply-connected quadrature domains \cite{CM}, 
or various related 
physical problems such as vortex dynamics or Stokes flow in 
multiply-connected two dimensional regions.

\vspace{1cm}

\smallskip
\noindent
{John \textsc{Gibbons}}\\
{Imperial College\\
180 Queen's Gate\\
London SW7 2BZ, UK\\
E-mail {j.gibbons@imperial.ac.uk}}

\smallskip
\smallskip

\noindent
{Shigeki  \textsc{Matsutani}}\\
{8-21-1 Higashi-Linkan,\\
Minami-ku, Sagamihara 252-0311,\\
JAPAN\\
E-mail {rxb01142@nifty.com}}

\smallskip
\smallskip
\noindent
{Yoshihiro \textsc{\^Onishi}}\\
{Department of Mathematics \\
Faculty of Education and Human Sciences, \\
University of Yamanashi,\\
4-3-11, Takeda, Kofu 400-8511, \\
JAPAN\\
E-mail {yonishi@yamanashi.ac.jp}}
\label{finishpage}

\end{document}